\documentclass{amsart}

\usepackage[T1]{fontenc}
\usepackage{amssymb}

\newtheorem{theorem}[subsection]{Theorem}
\newtheorem{proposition}[subsection]{Proposition}
\newtheorem{lemma}[subsection]{Lemma}

\newtheorem{claim}{Claim}
\newtheorem{subclaim}{Subclaim}
\newtheorem{question}[subsection]{Question}

\theoremstyle{definition}
\newtheorem{definition}[subsection]{Definition}

\theoremstyle{remark}
\newtheorem*{remark}{Remark}

\newcommand{\omegaone}{{\omega_1}}
\newcommand{\omegatwo}{{\omega_2}}

\newcommand{\ch}{\mathsf{CH}}
\newcommand{\ma}{\mathsf{MA}}
\newcommand{\zfc}{\mathsf{ZFC}}

\newcommand{\cohen}[1]{\mathsf{Add}(\omega, {#1})}
\newcommand{\cohengen}[2]{\mathsf{Add}({#1}, {#2})}

\newcommand{\set}[1]{\left\{ {#1} \right\}}
\newcommand{\partition}[2]{{#1} \to ({#1}, {#2})^2}
\newcommand{\partitiongeneral}[3]{{#1} \to ({#2}, {#3})^2}
\newcommand{\notpartition}[2]{{#1} \not\to ({#1}, {#2})^2}
\newcommand{\balancedpartition}[2]{{#1} \to ({#2})^2}
\newcommand{\notbalancedpartition}[2]{{#1} \not\to ({#2})^2}
\newcommand{\type}{\mathrm{type}}
\newcommand{\dom}{\mathrm{dom}}
\newcommand{\ot}{\mathrm{ot}}

\newcommand\restr[2]{{% we make the whole thing an ordinary symbol
  \left.\kern-\nulldelimiterspace % automatically resize the bar with \right
  #1 % the function
  \vphantom{\big|} % pretend it's a little taller at normal size
  \right|_{#2} % this is the delimiter
  }}

\title{Complete Bipartite Partition Relations in Cohen Extensions}
\author{Dávid Uhrik}
\address{Charles University, Faculty of Mathematics and Physics, Department of Algebra, Sokolovská 83, 186 75 Praha 8, Czech Republic}
\email{uhrik@karlin.mff.cuni.cz}
\address{Institute of Mathematics of the Czech Academy of Sciences, Žitná 25, 115 67 Prague 1, Czech Republic}
\email{uhrik@math.cas.cz}

\thanks{This work has been supported by Charles University Research Center program No.\ UNCE/SCI/022 and by the Academy of Sciences of the Czech Republic (RVO 67985840).}

\begin{document}
\maketitle

\begin{abstract}
We investigate the effect of adding $\omega_2$ Cohen reals on graphs on $\omega_2$, in particular we show that $\omega_2 \to (\omega_2, \omega : \omega)^2$ holds after forcing with $\mathsf{Add}(\omega, \omega_2)$ in a model of $\mathsf{CH}$. We also prove that this results is in a certain sense optimal as $\mathsf{Add}(\omega, \omega_2)$ forces that $\omega_2 \not\to (\omega_2, \omega : \omega_1)^2$.
\end{abstract}

\section{Introduction}
Since Ramsey's result \cite{ramsey1930} that every infinite graph contains an infinite clique or an infinite independent set the area of partition relations has been very active. However, any straightforward generalizations of his result are bound to fail, the classical result of Sierpiński \cite{sierpinski1933} states that: $\notbalancedpartition{2^\kappa}{\kappa^+}$. Nevertheless there are ways to generalize Ramsey's result.

Dushnik and Miller showed in \cite{dushnikmiller1941} that the relation $\partition{\kappa}{\omega}$ always holds. Our notation for partition relations is standard, the reader unfamiliar with the arrow notation can find definitions in the Notation subsection.

Most questions concerning partition relations on $\omegaone$ have been answered. Considering the order type of the clique Erdős and Rado \cite{erdosrado1956} improved the previous result to $\partition{\omegaone}{\omega+1}$. By a result of Hajnal \cite{hajnal1960} their result is optimal as $\ch$ implies $\notpartition{\omegaone}{\omega+2}$. The best possible relation $\partition{\omegaone}{\alpha}$ for any countable ordinal $\alpha$ may consistently hold as shown by Todorčević in \cite{todorcevic1983}.

The case of $\omegatwo$ is far from being resolved. Erdős and Rado's previous result extends also to $\omegatwo$, i.e.\ $\partition{\omegatwo}{\omega + 1}$ in $\zfc$, but they also showed \cite{erdosrado1956} that $\ch$ implies $\partition{\omegatwo}{\omega_1 + 1}$; on the other hand $\partition{\omegatwo}{\omega_1}$ already implies $\ch$. If on top of $\ch$ we further assume $2^{\omega_1}=\omegatwo$, we are limited by $\omega_1 + 1$, i.e.\ $\notpartition{\omegatwo}{\omega_1 + 2}$ holds. Rather surprisingly Laver \cite{laver1975} showed that $\ma + 2^\omega = \omegatwo$ implies $\notpartition{\omegatwo}{\omega + 2}$.

If we weaken the assumption on the homogeneous set there are further results. Baumgartner \cite{baumgartner1976} showed that after adding $\omegatwo$ Silver reals via countable support product to a model of $\ch$ we obtain a model where $\partition{\omegatwo}{\omega : \omegatwo}$.

Our results are in a similar vein as Baumgartner's, instead of looking for a clique in a graph with no cofinal independent set we will be searching only for bipartite graphs. Our results can be seen as showing how much of the partition relation $\partition{\omegatwo}{\omega_1}$ remains after adding Cohen reals. We will prove that $\partition{\omegatwo}{\omega : \omega}$ holds after forcing with $\cohen{\omegatwo}$ in a model of $\ch$. On the other hand the same conclusion as in the Silver model cannot hold after adding Cohen reals as $\cohen{\omegatwo} \Vdash \notbalancedpartition{\omega_2}{\omega : \omega_1}$.

Our main results are in a more general setting covering also the case of adding $\lambda^+$ many Cohen subsets of $\kappa$ for $\kappa, \lambda$ regular and its effect on the analogous partition relations. In particular we also get that $\partition{\omega_3}{\omega_1 : \omega_1}$ holds after adding $\omega_3$ Cohen subsets of $\omegaone$ to a model of $\mathsf{GCH}$.

\subsection*{Notation}
We use standard set theoretic notation. If $X$ is a set and $\mu$ a cardinal, then $[X]^\mu := \set{Y \subseteq X \mid |Y| = \mu}$. For two sets $X,Y$ we define $X \otimes Y := \set{\{x,y\} \mid x \in X \land y \in Y}$. If $X,Y$ are subsets of some ordered set, then $X<Y$ means that each element of $X$ lies below each element of $Y$. $\mathsf{H}(\kappa)$ denotes the collection of all sets hereditarily of cardinality less than $\kappa$; these sets will be used for constructing suitable elementary submodels, for more on this subject we refer the reader to \cite[Chapter 24.]{justweese1997}. A collection of sets forms a $\Delta$-system if there is a fixed set $r$ such that the intersection of any two sets in the collection is exactly $r$. By $\mathrm{ot}(X)$ we mean the order type of a well-ordered set $X$.

A graph $G$ is a pair $(V,E)$, where $V$ is an arbitrary set and $E \subseteq [V]^2$. A subset $X \subseteq V$ is complete if $[X]^2 \subseteq E$. An independent set in $G$ is a subset $X$ of $V$ such that $[X]^2 \cap E$ is empty. If the vertex set of the graph is well-ordered and $\alpha$ and $\beta$ are ordinals, then a subgraph (sometimes referred to as configuration) of type $(\alpha : \beta)$ is one whose vertex set is $A \cup B$, where $A$ has order type $\alpha$ and $B$ has order type $\beta$, $A<B$ and $A \otimes B \subseteq E$.

Given ordinals $\alpha, \beta, \gamma$ the partition relation $\partitiongeneral{\alpha}{\beta}{\gamma}$ is the statement that given a function $c: [\alpha]^2 \to 2$ there is a subset $X$ of $\alpha$ such that either the order type of $X$ is $\beta$ and $c''[X]^2 = \{0\}$ or the order type of $X$ is $\gamma$ and $c''[X]^2 = \{1\}$. The weaker relation $\partitiongeneral{\alpha}{\beta}{\gamma : \delta}$ says that for every function $c: [\alpha]^2 \to 2$ either there is an $X \subseteq \alpha$ such that the order type of $X$ is $\beta$ and $c''[X]^2 = \{0\}$ or there are sets $X,Y \subseteq \alpha$ such that $X<Y$, the order type of $X$ is $\gamma$, the order type of $Y$ is $\delta$ and $c''[X \otimes Y] = \{1\}$.

By a slight abuse of notation we will write $c(\alpha, \beta)$ instead of $c(\{\alpha, \beta\})$ and when writing $c(\alpha, \beta)$ we also tacitly assume that $\alpha < \beta$ if a natural ordering is present. The notation $\partitiongeneral{\alpha}{\beta}{\gamma}$ is shortened to $\balancedpartition{\alpha}{\beta}$, when $\beta = \gamma$.

Each function $c : [\alpha]^2 \to 2$ defines a graph on $\alpha$, namely $(\alpha, c^{-1}[\{1\}])$. Thus it will sometimes be convenient to talk about arbitrary functions on $\alpha$ and instead of looking for homogeneous sets for the coloring we can consider independent sets and cliques, i.e.\ we can rephrase the notion of partition relations as follows: for ordinals $\alpha, \beta, \gamma$ the partition relation $\partitiongeneral{\alpha}{\beta}{\gamma}$ says that given any graph whose vertex set is $\alpha$ and there is no independent set of order type $\beta$ we can find a complete subgraph of order type $\gamma$. The relation $\partitiongeneral{\alpha}{\beta}{\gamma : \delta}$ says that every graph on $\alpha$ either has an independent set of order type $\beta$ or a subgraph of type $(\gamma : \delta)$.

Suppose $\kappa\le\lambda$ are cardinals, the forcing for adding $\lambda$ many Cohen subsets of $\kappa$ will be denoted $\cohengen{\kappa}{\lambda}$, its underlying set is $\set{p : \lambda \to 2 \mid |p|<\kappa}$ and the ordering is reverse inclusion, see \cite{kunen1980} for an introduction to independence proofs in set theory.

\subsection*{Acknowledgement.} The author is grateful to David Chodounský, Chris Lambie-Hanson and Stevo Todorčević for helpful discussions on the topic which greatly improved the exposition of this paper.

\section{Positive result}

As the central tool of this section will be double $\Delta$-systems let us review a classical result about the existence of $\Delta$-systems. The proof can be found in \cite{kunen1980}.

\begin{theorem}[$\Delta$-system lemma]
Suppose $\kappa$ is an infinite cardinal, $\lambda > \kappa$ is regular, for each $\alpha < \lambda$ we have $|\alpha^{<\kappa}| < \lambda$ and $\mathcal{A}$ is a collection of sets such that $|\mathcal{A}| \ge \lambda$. If for all $x \in \mathcal{A}$ we have $|x|<\kappa$, then there is a $\mathcal{B} \subseteq \mathcal{A}$, such that $|\mathcal{B}| = \lambda$ and $\mathcal{B}$ forms a $\Delta$-system.
\end{theorem}

Double $\Delta$-systems were introduced by Todorčević in \cite{todorcevic1986} and utilized also in other papers. For an application in a similar context see \cite[Theorem 3.3.]{todorcevic2021}. A nice exposition on double $\Delta$-systems and their higher analogues can be found in \cite{lambiehanson2022}.

\begin{definition}
Let $\Gamma$ be a set of ordinals and $D := \set{p_{\alpha \beta} \mid \set{\alpha, \beta} \in [\Gamma]^2}$ a collection of sets. We say that $D$ is a \emph{double $\Delta$-system with root $p^0 \cup p^1$} if the following holds:
\begin{enumerate}
    \item for every $\alpha \in \Gamma$ $\set{p_{\alpha \beta} \mid \beta \in \Gamma \setminus (\alpha + 1)}$ is a $\Delta$-system with root $p^0_\alpha$
    \item for every $\beta \in \Gamma$ $\set{p_{\alpha \beta} \mid \alpha \in \Gamma \cap \beta}$ is a $\Delta$-system with root $p^1_\beta$
    \item $\{p^0_\alpha \mid \alpha \in \Gamma\}$ is a $\Delta$-system with root $p^0$
    \item $\{p^1_\beta \mid \beta \in \Gamma\}$ is a $\Delta$-system with root $p^1$
\end{enumerate}
\end{definition}

\begin{remark}
In our case the sets $p_{\alpha \beta}$ will be conditions in the Cohen forcing. The notation $p_{\alpha \beta}$ implicitly assumes that $\alpha < \beta$. Note also that the conditions on the double $\Delta$-system ensure that $\bigcap \{p_{\alpha\beta} \mid \{\alpha,\beta\} \in [\Gamma]^2\} = p^0 = p^1$.
\end{remark}

We will define the notion of isomorphism between forcing conditions.

\begin{definition}
Given $p,q \in \cohengen{\kappa}{\lambda}$ we define the set $\type(p)$ as the sequence $(p_i)_{i < \mu}$, where $\mu = \mathrm{ot}(\dom(p))$ and $(p_i)_{i < \mu}$ is an enumeration of the values of $p$ as a sequence respecting the ordering of its domain. Conditions $p,q$ are isomorphic, $p \simeq q$, if $\type(p) = \type(q)$.

The type of a pair, $\type(p,q)$, is defined again as a sequence $(s_i)_{i < \eta}$, where $\eta = \mathrm{ot}(\dom(p) \cup \dom(q))$ and if $(r_i)_{i < \eta}$ is an enumeration of $\dom(p) \cup \dom(q)$ respecting the ordering, then $s_i = (v_p^i, v_q^i)$, where if $r_i \in \dom(p)$, then $v_p^i = p(r_i)$, else $v_p^i = 2$; analogously for $v_q^i$. Two pairs of conditions $(p,q), (r,s)$ are isomorphic, $(p,q) \simeq (r,s)$, if $\type(p,q) = \type(r,s)$.
\end{definition}
\begin{remark}
The sets $\type(p)$ and $\type(p,q)$ just record all information about a condition (a pair of conditions). In other words it codes them as structures. Also note that $\type(p,q) = \type(r,s)$ implies the equality of types coordinate-wise, i.e.\ $\type(p) = \type(r)$ and $\type(q) = \type(s)$.
\end{remark}

In further applications we will need a more uniform version of $\Delta$-systems of conditions.

\begin{lemma}
    Suppose $\kappa < \lambda$ are regular cardinals, $|2^\mu| < \lambda$ for all $\mu < \kappa$ and $\set{p_\alpha \mid \alpha < \lambda}$ is a set of conditions in $\cohengen{\kappa}{\lambda^+}$ forming a $\Delta$-system. There is an $X \in [\lambda]^\lambda$ and an $s \subseteq \mathrm{ot}(\dom(p_0))$ such that:
    \begin{enumerate}
        \item for all $\alpha, \beta \in X$ we have $\type(p_\alpha) = \type(p_\beta)$, and
        \item for all $\alpha \in X$ if $(d_i \mid i < \mathrm{ot}(\dom(p_\alpha)))$ is an increasing enumeration of the domain of $p_\alpha$, then $\{ (d_i, p_\alpha(d_i)) \mid i \in s \}$ is exactly the root of the original $\Delta$-system.
    \end{enumerate}
\end{lemma}

\begin{proof}
    The proof is a routine counting argument.

    To ensure that the types of all the conditions are the same note that the order type of the domain of any condition from $\cohengen{\kappa}{\lambda^+}$ is an ordinal below $\kappa$. Let $f:\lambda \to \kappa$ be a function such that $f(\alpha)=\ot(\dom(p_\alpha))$. As $\kappa < \lambda$ and $\lambda$ is regular we get a $\gamma_0 < \kappa$ and an $A \in [\lambda]^\lambda$ such that the order type of the domain of $p_\alpha$ is $\gamma_0$ for all $\alpha \in A$. Next consider each function from $\gamma_0$ to $2$. As $2^{\gamma_0} < \lambda$, there is less than $\lambda$ many such functions. Given a condition $p_\alpha$ for $\alpha \in A$ let $\varphi_{\alpha}: \gamma_0 \to \dom(p_\alpha)$ be the unique increasing bijection and define a function $g: \lambda \to 2^{\gamma_0}$ such that $g(\alpha) = p_\alpha \circ \varphi_\alpha$. As before there is an $A' \in [A]^\lambda$ and a fixed function $q: \gamma_0 \to 2$ such that $\type(p_\alpha) = q$ for each $\alpha \in A'$.

    To make sure that the relative position of the root of the $\Delta$-system stays the same across all conditions define another function $h: \lambda \to 2^{\gamma_0}$ such that $h(\alpha)(\beta) = 1$ if and only if $(\varphi_\alpha(\beta), p_\alpha(\varphi_\alpha(\beta))$ is in the root of the $\Delta$-system. Analogously as before we find a function $r: \gamma_0 \to 2$ such that for $\lambda$ many $\alpha$ we have $h(\alpha) = r$, this ensures the second condition.
\end{proof}

\begin{lemma} \label{delta_iso}
Suppose $\gamma_0 \le \gamma_1$ are ordinals and $\set{p_\alpha \mid \alpha < \gamma_0}$,  $\set{q_\alpha \mid \alpha < \gamma_1}$ are sets of conditions in $\cohengen{\kappa}{\lambda}$. If $\set{q_\alpha \mid \alpha < \gamma_1}$ forms a $\Delta$-system and for each $\alpha < \beta < \gamma_0$ we have $(p_\alpha, p_\beta) \simeq (q_\alpha, q_\beta)$, then $\set{p_\alpha \mid \alpha < \gamma_0}$ also forms a $\Delta$-system.
\end{lemma}

\begin{proof}
First enumerate in increasing order the domain of $q_0$ as $(d_i \mid i < \mathrm{ot}(\dom(q_0)))$. As the conditions $\set{q_\alpha \mid \alpha < \gamma_1}$ are isomorphic and form a $\Delta$-system let $s$ be the set of indices $i < \mathrm{ot}(\dom(q_0))$ such that $\set{(d_i, q_0(d_i)) \mid i \in s}$ is exactly the root.

If we similarly enumerate the domain of $p_0$ as $(e_i \mid i < \mathrm{ot}(\dom(p_0)))$ (note that $\mathrm{ot}(\dom(q_0)) = \mathrm{ot}(\dom(p_0))$), we claim that $\set{(e_i, p_0(e_i)) \mid i \in s}$ is the root of the $\Delta$-system formed by the conditions $\set{p_\alpha \mid \alpha < \gamma_0}$.

Given any $p_\alpha$ and $p_\beta$ as this pair is isomorphic to the pair $(q_\alpha, q_\beta)$ we have that $d_i \in \dom(q_\alpha) \cap \dom(q_\beta)$ if and only if $e_i \in \dom(p_\alpha) \cap \dom(p_\beta)$ and this happens exactly when $i \in s$, also when $d_i \in \dom(q_\alpha) \cap \dom(q_\beta)$ then $d_i$ is also the $i$-th element of the domain of both $q_\alpha$ and $q_\beta$ and the same holds for $e_i$ and any $p_\alpha$ and $p_\beta$. Finally as the conditions $\set{q_\alpha \mid \alpha < \gamma_1}$ are isomorphic so are $\set{p_\alpha \mid \alpha < \gamma_0}$ so in particular $p_\alpha \simeq p_0 \simeq p_\beta$ and we are done.
\end{proof}

The main theorem follows.

\begin{theorem} \label{cohenandgraphs}
Suppose $\kappa < \lambda$ are regular cardinals. If $\lambda^{<\lambda} = \lambda$ and for each $\alpha < \lambda$ we have $|\alpha^{<\kappa}| < \lambda$, then $\cohengen{\kappa}{\lambda^+}$ forces the relation $\partition{\lambda^+}{\mu : \mu}$ for any $\mu < \kappa^+$.
\end{theorem}

\begin{proof}

Consider the extension by the Cohen forcing adding $\lambda^+$ subsets of $\kappa$. Fix a condition $q$ and a name $\dot{c}$ such that $q$ forces that in the extension $c$ is a function from $[\lambda^+]^2$ to $2$. Without loss of generality we will assume that $q=\emptyset$. If it is the case that $$\emptyset \Vdash \exists X \in [\lambda^+]^{\lambda^+} : \dot{c}''[X]^2 = \{0\}$$ we are done, so suppose this is not the case. Now an improved double $\Delta$-system can be found.

The proof of the following claim closely resembles the argument in \cite[Theorem~3.3.]{todorcevic2021}, where Todorčević constructs a double $\Delta$-system of Cohen conditions with analogous properties but he considers colorings with range $\omega$. Todorčević used this technique already in \cite{todorcevic1986}.

\begin{claim}
There is a set $X \in [\lambda^+]^{\lambda}$ of order type $\lambda$ and a set of conditions  $D := \set{p_{\alpha \beta} \in \cohengen{\kappa}{\lambda^+} \mid \set{\alpha, \beta} \in [X]^2}$ such that the following holds:

\begin{enumerate}
    \item $p_{\alpha \beta} \Vdash \dot{c}(\alpha, \beta) = 1$,
    \item $D$ forms a double $\Delta$-system.
\end{enumerate}
\end{claim}

\begin{proof}
In the ground model for every $\alpha < \beta$ in $\lambda^+$ either $\emptyset \Vdash \dot{c}(\alpha, \beta) = 0$ or there is a condition $p$ such that $p \Vdash \dot{c}(\alpha, \beta) = 1$. For every pair fulfilling the second option fix such a condition $p_{\alpha \beta}$, otherwise put $p_{\alpha \beta} := \emptyset$. Consider a regular cardinal $\theta$ large enough so that $\mathsf{H}(\theta)$ contains all the relevant objects we have considered so far. Choose an elementary submodel $M$ of $\mathsf{H}(\theta)$ of size $\lambda$ such that $M^{<\lambda} \subseteq M$ (we assume $\lambda^{<\lambda} = \lambda$ in the ground model) and $\delta := M \cap \lambda^+$ has cofinality $\lambda$. Fix also a $\lambda$-sequence converging to $\delta$, say $(d_\alpha \mid \alpha < \lambda)$.

\begin{subclaim}
There is a set, $B$, cofinal in $\delta$ of order type $\lambda$ with the following properties for every $\alpha < \beta < \gamma$ in $B$:
\begin{enumerate}
    \item \label{cond_iso} $p_{\alpha \beta} \simeq p_{\alpha \delta}$
    \item \label{cond_restr} $p_{\alpha \beta}\restriction\beta = {p_{\alpha \delta}}\restriction{\delta}$
    \item \label{cond_iso_pair} $(p_{\alpha \gamma}, p_{\beta \gamma}) \simeq (p_{\alpha \delta}, p_{\beta \delta})$
    \item \label{cond_dom} $\dom({p_{\alpha \beta}}) \subseteq \gamma$
    \item \label{cond_forc} $p_{\alpha \beta} \Vdash \dot{c}(\alpha, \beta) = 1$
    \item \label{cond_forc_delta} $p_{\alpha \delta} \Vdash \dot{c}(\alpha, \delta) = 1$
\end{enumerate}
\end{subclaim}
\begin{proof}
The set $B$ cannot be an element of $M$ but any initial segment of such a set $B$ belongs to $M$ because $M$ is closed under sequences of length $<\lambda$, this will be used in the inductive construction. Suppose we have constructed an initial segment of $B$, a sequence $b := (b_\xi \mid \xi < \beta)$ for some ordinal $\beta < \lambda$ satisfying all the conditions, and $b_\alpha \ge d_\alpha$ for all $\alpha < \beta$. We want to construct $b_\beta$ above $d_\beta$. Note that $b \in M$ as the sequence has length $<\lambda$, and also $d_\beta \in M$. Consider the following sequences/matrices for every $\eta \in \lambda^+ \setminus d_\beta$:

\begin{align*}
    S_1^\eta &:= (\type(p_{\alpha \eta}) \mid \alpha \in b) \\
    S_2^\eta &:= ({p_{\alpha \eta}}\restriction{\eta} \mid \alpha \in b) \\
    S_3^\eta &:= (\type(p_{\alpha \eta}, p_{\beta \eta}) \mid \alpha < \beta \in b)
\end{align*}

For any $\eta \in M$ all three sets $S_1^\eta, S_2^\eta$ and $S_3^\eta$ are definable in $M$ and moreover the entire sequence $(S^\eta_i \mid \eta < \lambda^+)$ is in $M$ for each $i$ in $\{1,2,3\}$. Note also that $S_i^\delta$ is an element of $M$ for $i\in\set{1,2,3}$ even though $\delta \not\in M$, this follows from $M$ being closed under sequences of length $<\lambda$. Now define the set:
\begin{multline*}
    S := \{ \eta < \lambda^+ \mid \eta \ge d_\beta \land S_1^\eta = S_1^\delta \land S_2^\eta = S_2^\delta \land S_3^\eta = S_3^\delta \land \\ \land \eta > \sup\set{\sup\dom(p_{\alpha \beta}) \mid \alpha < \beta \in b} \land \forall \alpha \in b : p_{\alpha \eta} \Vdash \dot{c}(\alpha, \eta) = 1 \}
\end{multline*}
It is obvious that $S$ is definable in $M$ and $\delta \in S$, thus $S$ is a stationary subset of $\lambda^+$.

Finally consider the set
$$
T := \set{\eta \in S \mid \forall \xi \in S \cap \eta : \emptyset \Vdash \dot{c}(\xi, \eta) = 0}
$$

If $\delta$ is not an element of $T$, then there is some $\xi \in S \cap \delta$ for which $\emptyset \not\Vdash \dot{c}(\xi, \delta) = 0$ so $p_{\xi \delta}$ was defined as some condition $p$ such that $p \Vdash \dot{c}(\xi, \delta) = 1$, as $\xi$ is also an element of $S$ we can put $b_\beta := \xi$. Now each condition of the claim is satisfied as witnessed by $\xi$ belonging to $S$ and the fact that $\xi$ witnesses that $\delta \not\in T$.

On the other hand if $\delta \in T$, then $T$ is unbounded in $\lambda^+$ and clearly $\emptyset \Vdash \dot{c}''[T]^2 = \{0\}$, a contradiction.
\end{proof}

Let $B$ be the set constructed in the previous subclaim. We can also assume that $\set{p_{\alpha \delta} \mid \alpha \in B}$ is a $\Delta$-system (we assume that for each $\alpha < \lambda$ we have $|\alpha^{<\kappa}| < \lambda$) so in particular $p_{\alpha \delta} \simeq p_{\beta \delta}$ for $\alpha, \beta \in B$. We now show that the set $B$ can be refined so that the set of conditions $\set{p_{\alpha \beta} \mid \set{\alpha, \beta} \in [B]^2}$ will form a double $\Delta$-system.

The previous paragraph, condition (\ref{cond_iso_pair}) and Lemma \ref{delta_iso} imply that for each $\gamma$ also the set $\set{p_{\alpha \gamma} \mid \alpha \in B \cap \gamma}$ is a $\Delta$-system with root $p^1_\gamma$. We can now assume that $\set{p^1_\gamma \mid \gamma \in B}$ also forms a $\Delta$-system with root $p^1$.

Conditions (\ref{cond_iso}), (\ref{cond_restr}) and (\ref{cond_dom}) imply that for each $\alpha \in B$ the set of conditions $\set{p_{\alpha \beta} \mid \beta \in B \setminus (\alpha + 1)}$ is a $\Delta$-system with root $p^0_\alpha := {p_{\alpha \delta}}\restriction{\delta}$. Given any $p_{\alpha \beta}$ and $p_{\alpha \gamma}$ for $\alpha < \beta < \gamma$ in $B$ consider the intersection $p_{\alpha \beta} \cap p_{\alpha \gamma}$, clearly ${p_{\alpha \delta}}\restriction{\delta} \subseteq p_{\alpha \beta} \cap p_{\alpha \gamma}$ by condition~(\ref{cond_restr}). For the other direction if $(d,v) \in p_{\alpha \beta} \cap p_{\alpha \gamma}$, then $d < \gamma$ by condition (\ref{cond_dom}) thus $(d,v) \in p_{\alpha \gamma}\restriction\gamma$ and again by condition (\ref{cond_restr}) $(d,v) \in p_{\alpha \delta}\restriction{\delta}$. Finally we can also assume that $\set{p^0_\alpha \mid \alpha \in B}$ forms a $\Delta$-system with root $p^0$. Let $X$ be the refined set $B$, this is our desired set.
\end{proof}

We will denote the root $p^0 = p^1$ of the double $\Delta$-system simply as $p$.

Before we proceed fix an ordinal $\kappa \le \mu < \kappa^+$. Choose two sets: $X_0$ and $X_1$, subsets of $X$ such that $X_0 < X_1$ and the order type of both sets is $\kappa \cdot \mu$. Fix also a bijection $g: \kappa \to \mu$.

\begin{claim}
$p$ forces a $(\mu : \mu)$ configuration in color $1$.
\end{claim}

\begin{proof}
Let $G$ be a generic set containing $p$, by induction we will construct sequences $(s_\alpha \mid \alpha < \kappa)$ in $X_0$ and $(t_\alpha \mid \alpha < \kappa)$ in $X_1$ such that $s_\alpha$ is in the $g(\alpha)$-th section of $X_0$, i.e.\ if $f: \kappa \cdot \mu \to X_0$ is the unique increasing bijection then $$s_\alpha \in [f(\kappa \cdot g(\alpha)), f(\kappa \cdot (g(\alpha)+1))),$$ denote this subset of $X_0$ as $X_0^\alpha$, analogously for $t_\alpha$ and $X_1$. We will make sure that for all $\alpha, \beta \in \kappa$ we have $p_{s_\alpha t_\beta} \in G$; as $p_{s_\alpha t_\beta}$ forces the color of the pair $\{s_\alpha, t_\beta\}$ to be $1$ this will ensure the conclusion of the claim.

To start the induction note that by genericity for some $\alpha \in X_0^0$ we have $p^0_\alpha \in G$, this is because $\set{p^0_\alpha \mid \alpha \in X_0^0}$ is a $\Delta$-system of size $\kappa$ with root $p^0 \ge p$ and thus this set is predense below $p$. By the same argument there is some $\beta \in X_1^0$ such that $p_{\alpha \beta}$ is in $G$, so put $s_0 := \alpha$ and $t_0 := \beta$.

Suppose we have already constructed $(s_\alpha \mid \alpha < \gamma)$ and $(t_\alpha \mid \alpha < \gamma)$ such that $p_{s_\alpha t_\beta} \in G$ for all $\alpha, \beta < \gamma$. We will now find $\sigma \in X_0^\gamma$ such that $\set{p_{\sigma t_\alpha} \mid \alpha < \gamma} \subseteq G$ and this will be our $s_\gamma$.

Suppose that no $\sigma$ satisfies our requirements, i.e.\ there is no $\sigma$ in $X_0^\gamma$ such that $\set{p_{\sigma t_\alpha} \mid \alpha < \gamma} \subseteq G$. Then there must exist a condition $r \le p$ forcing this (note that $\set{p_{\sigma t_\alpha} \mid \alpha < \gamma}$ is an element of the ground model because our forcing is $\kappa$-closed): $$ r \Vdash \forall \sigma \in X_0^\gamma : \{p_{\sigma t_\alpha} \mid \alpha < \gamma\} \not\subseteq \dot{G}$$ This means that for all $\sigma \in X_0^\gamma$ there exists a $\beta < \gamma$ such that $r \bot p_{\sigma t_\beta}$. By going to a refinement we can assume that for $\kappa$ many $\sigma$ there is a fixed $\beta' < \gamma$ such that $r \bot p_{\sigma t_{\beta'}}$, call this set $C$. However note that the set $\{p_{\sigma t_{\beta'}} \mid \sigma \in C\}$ is a $\Delta$-system with root $p^1_{t_{\beta'}}$ which is contained in the generic set $G$ because $p_{s_0 t_{\beta'}} \le p^1_{t_{\beta'}}$ and $p_{s_0 t_{\beta'}} \in G$. Since $r$ has size $<\kappa$ and $r \parallel p^1_{t_{\beta'}}$, it cannot be incompatible with every condition from $\{p_{\sigma t_{\beta'}} \mid \sigma \in C\}$, a contradiction.

The construction of $t_\gamma$ is almost verbatim.

\end{proof}
This concludes the proof.
\end{proof}

\section{A Negative partition relation from Cohen forcing}

The result of the previous section cannot be strengthened so that the second partition of the bipartite graph has size $\kappa^+$.

\begin{proposition}
If $\kappa < \lambda$ are regular cardinals, then $\cohengen{\kappa}{\lambda} \Vdash \notbalancedpartition{\lambda}{\kappa : \kappa^+}$.
\end{proposition}

\begin{proof}
Consider an equivalent form of the forcing notion, specifically the poset $C_{S} : = \set{p : S \to 2 \mid |p| < \kappa}$, where $S := [\lambda]^2$ and the ordering is reverse inclusion. We will prove that the generic graph, the union over the generic set $G$, added this way does not contain a homogeneous $(\kappa : \kappa^+)$ configuration.

Suppose, for contradiction, that in the extension there is a set $X$ of size $\kappa$ and a set $Y$ of size $\kappa^+$ above it so that that all edges between them are monochromatic. Use the fact that when forcing with $\cohengen{\kappa}{\lambda}$ any set of size $\kappa$ can be decided already when forcing over a domain of size $\kappa$ \cite[Lemma VIII.2.2.]{kunen1980}. To be more precise denote by $M$ the ground model; there is a set $I \subseteq [\lambda]^2$ of size $\kappa$ so that $X \in M[G \cap C_{I}]$ (note that $C_{[\lambda]^2} \cong C_{I} \times C_{[\lambda]^2 \setminus I}$). Now working in the extension by $C_I$, there must exist a condition $p$ in $ C_{[\lambda]^2 \setminus I}$ so that $p \Vdash y \in \dot{Y}$ for some $y \not\in \bigcup I$, otherwise $C_{[\lambda]^2 \setminus I} \Vdash \dot{Y} \subseteq \bigcup I$, which is not possible. Now $p$ has size $<\kappa$ and $|X|=\kappa$ hence there must be an $x \in X \setminus \bigcup \dom(p)$. Now $p$ can be extended by the pair $(\set{x, y}, i)$ for both $i \in \set{0, 1}$ which is a contradiction.
\end{proof}

\section{Possible strengthening of our result}
The consistency of the relation $\partition{\omegatwo}{\omega + 2}$ with $\neg\ch$ is unknown and seems to be substantially more involved than our result. By a result of Raghavan and Todorčević \cite{raghavantodorcevic2018} this relation implies the non-existence of $\omegatwo$-Suslin trees. Laver showed \cite{laver1975} that in a model where $\ma$ holds and $2^\omega = \omegatwo$ we have $\notpartition{\omegatwo}{\omega : 2}$.

\begin{question}
Is the relation $\partition{\omegatwo}{\omega + 2}$ consistent with $\neg\ch$?
\end{question}

\end{document}